\documentclass[11pt]{amsart}

\title{Quadratic Points on Non-Split Cartan Modular Curves}
\author{Philippe Michaud-Rodgers}
\usepackage{geometry}
\usepackage{amsmath} 
\usepackage{amssymb} 
\usepackage{mathtools}
\usepackage{framed}
\usepackage{pifont} 
\usepackage{enumerate} 
\usepackage{wasysym} 
\usepackage{commath}
\usepackage{enumitem}
\usepackage{graphicx}
\usepackage{framed}
\usepackage{amsmath}
\usepackage{amssymb}
\usepackage{amsfonts}
\usepackage{mathrsfs}
\usepackage{mathtools}
\usepackage{url}
\usepackage{array}
\usepackage{amsthm}
\usepackage[english]{babel}
\usepackage[utf8]{inputenc}
\newtheorem{theorem}{Theorem}[section]
\newtheorem*{theorem*}{Theorem}

\newtheorem {corollary}[theorem]{Corollary}
\newtheorem {lemma}[theorem]{Lemma}
\newtheorem {proposition}[theorem]{Proposition}
\theoremstyle{definition}

\theoremstyle{remark}

\newtheorem*{remark*}{Remark}
\usepackage{yfonts}
\usepackage{amsmath,amssymb,tikz-cd}

\usepackage{etoolbox}
\apptocmd{\sloppy}{\hbadness 10000\relax}{}{}

\DeclareMathOperator{\SL}{SL}
\usepackage[numbers]{natbib}
\usepackage{cite}
\usepackage{mathtools, nccmath}
\newcommand{\hooklongrightarrow}{\lhook\joinrel\longrightarrow}

\providecommand{\Q}{\mathbb{Q}}
\providecommand{\Z}{\mathbb{Z}}

\providecommand{\R}{\mathbb{R}}

\renewcommand{\arraystretch}{2.0}

\address{Mathematics Institute\\
    University of Warwick\\
   Coventry  CV4 7AL \\
    United Kingdom}

\email{p.rodgers@warwick.ac.uk}

\begin{document}
\begin{abstract}
In this paper we study quadratic points on the non-split Cartan modular curves $X_{ns}(p)$, for $p = 7, 11,$ and $13$. Recently, Siksek proved \citep{samirnote} that all quadratic points on $X_{ns}(7)$ arise as pullbacks of rational points on $X_{ns}^+(7)$. Using similar techniques for $p=11$, and employing a version of Chabauty for symmetric powers of curves for $p=13$, we show that the same holds for $X_{ns}(11)$ and $X_{ns}(13)$. As a consequence, we prove that certain classes of elliptic curves over quadratic fields are modular.
\end{abstract}
\maketitle
{\centering\footnotesize  \emph{Pour Pap{\'e}} 
\par}
\setcounter{tocdepth}{1}
\tableofcontents
\vspace{-5mm}
\section{Introduction}

Rational points on many classes of modular curves have been, and continue to be, studied extensively. Recently, there has been an increased interest in quadratic points on modular curves. Quadratic points on $X_1(N)$ are well understood: Kamienny proved in \citep{kamienny} that there are no quadratic points on $X_1(N)$ for $N \geq 17$. Also, due to several recent papers \citep{ozmansiksek, joshaquad, hyperquad}, progress has been made in understanding quadratic points on $X_0(N)$: all quadratic points on $X_0(N)$ have been classified for genus  $2,3,4,$ and $5$. In this paper, we aim to understand the quadratic points on some non-split Cartan modular curves; namely $X_{ns}(p)$, for $p=7,11,$ and $13$.

Quadratic points on algebraic curves often arise due to the existence of a degree $2$ map to a curve with rational points. In fact, a well-known result of Harris and Silverman \citep[Corollary 3]{HandS} states that a curve of genus $\geq 2$ can have infinitely many quadratic points only if it is hyperelliptic or bielliptic. The curve $X_{ns}(p)$ comes equipped with a degree $2$ map to the modular curve $X_{ns}^+(p)$. Rational points on $X_{ns}^+(p)$ therefore provide a source of quadratic points on $X_{ns}(p)$. Any quadratic point that does \emph{not} arise in this way is said to be \emph{exceptional}. The main result of this paper is the following, with the case $p=7$ due to Siksek \citep{samirnote}.

\begin{theorem}[Main theorem]\label{mainthm} There are no exceptional quadratic points on $X_{ns}(p)$ for $p = 7, 11,$ or $13$.
\end{theorem}

The curves $X_{ns}^+(p)$ for $p=7$ and $11$ have infinitely many rational points. Indeed, $X_{ns}^+(7)$ is isomorphic to $\mathbb{P}^1$ and $X_{ns}^+(11)$ is an elliptic curve with Mordell--Weil group isomorphic to $\mathbb{Z}$, so we obtain infinitely many quadratic points on both $X_{ns}(7)$ and $X_{ns}(11)$, all non-exceptional. In contrast, $X_{ns}^+(13)$ has precisely seven rational points \citep{cursed}, and so (by the theorem above) $X_{ns}(13)$ has precisely seven pairs of quadratic points. Moreover, by considering equations for the degree $2$ map from $X_{ns}(13)$ to $X_{ns}^+(13)$, we are able to list (Table 2) these quadratic points on our model for $X_{ns}(13)$. 

A consequence of this theorem is the following.

\begin{corollary} Let $E$ be an elliptic curve defined over a quadratic field $K$, and suppose that $\overline{\rho}_{E,p}(\mathrm{Gal}(\overline{K}/K)) \subseteq C_{ns}(p)$ for $p=7,11,$ or $13$. Then $j(E) \in \Q$. Thus $E$ is modular. 
 \end{corollary}

Here, $C_{ns}(p)$ is a non-split Cartan subgroup, and $\overline{\rho}_{E,p}$ is the mod $p$ Galois representation attached to $E$. 

\begin{proof} The elliptic curve $E$ gives rise to a non-cuspidal $K$-point (i.e. a non-cuspidal quadratic point) on $X_{ns}(p)$ by appropriate choice of level structure. By Theorem \ref{mainthm} the image of this point under the map $\varrho$ is a rational point, and so $j(E) \in \Q$ as the $j$-map is rational. We then know that $E$ is modular by \citep[Theorem A]{modularityQ}.
\end{proof}

We note that quadratic points on certain modular curves are studied to prove that all elliptic curves over real quadratic fields are modular \citep{FlHS}. These types of modularity results play a key role in the study of Diophantine equations.

We present here the outline for the rest of the paper. In Section 2 we review the construction of non-split Cartan modular curves, study their Jacobians, and start considering quadratic points. In Section 3 we communicate the unpublished proof of Siksek \citep{samirnote}, showing that Theorem \ref{mainthm} holds in the case $p=7$. We apply similar ideas in Section 4 to show the theorem also holds for $p=11$. In Section 5 we study the case $p=13$. This requires more sophisticated machinery, and we see how a version of Chabauty for symmetric powers of curves, as developed in \citep{Symmchab, joshaquad}, is used to prove the theorem in this case.

\bigskip
 
All the computer calculations are carried out using  \texttt{Magma} \citep{magma}. The \texttt{Magma} code used to support the computations in this paper (which uses parts of the code associated with the papers \citep{joshaquad, ozmansiksek}) can be found at:

\vspace{2pt}

\begin{center}
\url{https://warwick.ac.uk/fac/sci/maths/people/staff/michaud/c/}
\end{center}

\bigskip

I would like to express my sincere gratitude to Samir Siksek and Damiano Testa for their wonderful support. I would also like to thank the anonymous referee for a very careful reading of the paper, and for providing many helpful comments.  The author is supported by an EPSRC studentship (2274692).

\section{Non-Split Cartan Modular Curves}

\subsection{Construction of $X_{ns}(p)$ and $X_{ns}^+(p)$}
We briefly recall here the construction of $X_{ns}(p)$ and $X_{ns}^+(p)$, following \citep[p.98]{doublecover}. Given any positive integer $N$, and a subgroup $H \subseteq \mathrm{GL}_2(\Z / N\Z)$, we associate a congruence subgroup, $\Gamma_H$, to $H$, by defining  \[ \Gamma_H \coloneqq \{A \in \SL_2(\Z) : (A \text{  mod } N) \in \SL_2(\Z / N \Z ) \cap H \}. \] We then define the open modular curve $Y_H \coloneqq \Gamma_H \backslash \mathbb{H}$, and its compactification by adjoining cusps $X_H \coloneqq \Gamma_H \backslash \mathbb{H}^*$, where $\mathbb{H}^* \coloneqq \mathbb{H} \cup \mathbb{P}^1(\Q)$ denotes the extended upper half plane. Then $X_H$ is a compact Riemann surface and so it corresponds to an algebraic curve. If $H_1,H_2$ are conjugate subgroups of $\mathrm{GL}_2(\Z / N\Z)$ then $X_{H_1} \cong X_{H_2}$. If the determinant map $\det: H \mapsto (\Z / N\Z)^*$ is surjective, then $X_H$ is defined over $\Q$.

Now, let $p$ be an odd prime, and choose $\lambda \in \mathbb{F}_p$ a quadratic nonresidue. We define a \emph{non-split Cartan subgroup} as \[C_{ns}(p)\coloneqq \Big\lbrace \left( \begin{smallmatrix} \alpha & \beta \lambda \\ \beta & \alpha \end{smallmatrix} \right) : (\alpha,\beta) \in \mathbb{F}_p^2~ \backslash \{(0,0)\} \Big\rbrace \subseteq \mathrm{GL}_2(\Z / p\Z). \] This subgroup is defined up to conjugation: a different choice of $\lambda$ will lead to a conjugate subgroup. The normaliser of $C_{ns}(p)$ in $\mathrm{GL}_2(\Z / p\Z)$ is denoted by $C_{ns}^+(p)$ and is given by \[ C_{ns}^+(p) = \Big\lbrace \left( \begin{smallmatrix} \alpha & \beta \lambda \\ \beta & \alpha \end{smallmatrix} \right), \left( \begin{smallmatrix} \alpha & \beta \lambda \\ -\beta & -\alpha \end{smallmatrix} \right): (\alpha,\beta)  \in \mathbb{F}_p^2 ~\backslash \{(0,0)\} \Big\rbrace \subseteq \mathrm{GL}_2(\Z / p\Z). \] This subgroup is again defined up to conjugation. As both $C_{ns}(p)$ and $C_{ns}^+(p)$ are defined up to conjugation, we can associate modular curves to these subgroups. We have \[ X_{ns}(p)  \coloneqq X_{C_{ns}(p)} \qquad  \text{and} \qquad X_{ns}^+(p)  \coloneqq X_{C_{ns}^+(p)}.  \] We will refer to both these modular curves, rather lazily, as \emph{non-split Cartan modular curves}. Since for both of these subgroups the determinant map is surjective, the modular curves $X_{ns}(p)$ and $X_{ns}^+(p)$ are both defined over $\Q$.

Since $[C_{ns}^+(p) : C_{ns}(p)]=2$, this leads to a natural degree $2$ map  \[ \varrho:X_{ns}(p) \longrightarrow X_{ns}^+(p) \] We refer to the map $\varrho$ as a \emph{degeneracy map} as it simply uses a coarser equivalence relation.

The curve $X_{ns}(p)$ comes equipped with an automorphism, called the \emph{modular involution}, which we denote $w_p$. This is the map that interchanges the points of $\varrho^{-1}(P)$ for each $P \in X_{ns}^+(p)$. The curve $X_{ns}^+(p)$ is $X_{ns}(p) / \langle w_p \rangle$, and $\varrho$ is the quotient map. We note that it was recently shown \citep[Corollary 5.17]{autcart} that for all $p \geq 13$, $\mathrm{Aut}(X_{ns}(p))=\langle w_p \rangle$, and also that  $\mathrm{Aut}(X_{ns}^+(p))=\{1\}$.

For further information on non-split Cartan subgroups and their normalisers, as well as their definitions for composite $n$, we refer to \citep[pp.~2754-2756]{Barannorm} and \citep{autcart}.

\subsection{Jacobians and Chen's Isogenies}

The Jacobians of the curves $X_{ns}(p)$ and $X_{ns}^+(p)$, which we denote $J_{ns}(p)$ and $J_{ns}^+(p)$ respectively, are related to the Jacobian of $X_0(p^2)$. This is useful as we understand this Jacobian well. We recall here how $J_0(N)$ is isogenous to a product of abelian varieties associated to newforms, following \citep[p.~234]{kani}.

Given $N>0$, write $f_1, \dots, f_k$ for representatives of the Galois-conjugacy classes of Hecke eigenforms in the space of cuspforms $S_2(N)$. Write $K_i$ for the Hecke eigenfield of $f_i$. This is a totally real number field, and we write $d_i$ for its degree. To each $f_i$ is associated a simple abelian variety $A_i / \Q$ of dimension $d_i$, with endomorphism algebra $\mathrm{End}(A_i) \otimes \Q= K_i$. In particular, the rank of $A_i$ is a multiple of $d_i$. Each $f_i$ arises from a newform at some level $M_i \mid N$. Write $m_i$ for the number of divisors of $N/M_i$. Then  \[ J_0(N) \sim A_1^{m_1} \times \dots \times A_k^{m_k}, \] where $\sim$ denotes isogeny over $\Q$. It follows that $\mathrm{Rk}(J_0(N)) = \sum_{i=1}^k m_i \cdot  \mathrm{Rk}(A_i).$

Although finding the rank of some $A_i$ may not be feasible in general, it is sometimes possible to determine  whether or not the rank is zero. Write $L(A_i,s)$ for the $L$-function of $A_i$. Then a theorem of Kolyvagin and Logachev \citep{klthm} asserts that if $L(A_i,1) \neq 0$, then $\mathrm{Rk}(A_i)=0$. Using the modular symbols algorithms of Cremona \citep[pp.~29-31]{cremona} and Stein \citep[pp.~52-57]{stein} it is possible to check whether or not $L(A_i,1)$ is zero. This is implemented in \texttt{Magma} as part of the `modular abelian varieties package'.

Some of the abelian varieties in this decomposition are attached to the newforms at level $N$ and the others come from oldforms. We denote by $J_0(N)_{\mathrm{new}}$ the product of those abelian varieties coming from newforms and we denote the product of the remaining abelian varieties (multiplicity included) by $J_0(N)_{\mathrm{old}}$, so that \[J_0(N) \sim J_0(N)_{\mathrm{new}} \times J_0(N)_{\mathrm{old}}.\]

The following result links the Jacobians of non-split Cartan modular curves to the Jacobians of the curves $X_0(N)$ and  $X_0^+(N) \coloneqq X_0(N) / \langle u_N \rangle $, where $u_N$ denotes the Atkin--Lehner involution on $X_0(N)$. 
\begin{theorem}[Chen's Isogenies \citep{chenjac,autold}]\label{Chen} Write $J_0(N)$ and $J_0^+(N)$ for the Jacobians of $X_0(N)$ and $X_0^+(N)$ respectively. Let $p$ be a prime, then \begin{align*} J_{ns}(p) & \sim \; J_0(p^2)_{\mathrm{new}}, \\ J_{ns}^+(p) & \sim J_0^+(p^2)_{\mathrm{new}}. \end{align*}
\end{theorem} 

Note that $J_0^+(N) = J_0(N) / \langle u_N \rangle$, where we view $u_N$ as an endomorphism of $J_0(N)$ under the inclusion $\mathrm{Aut}(X_0(N)) \hookrightarrow \mathrm{End}(J_0(N))$. Similarly, $J_{ns}^+(p) = J_{ns}(p) / \langle w_p \rangle$. 

Of particular interest is the case when the ranks of $J_{ns}(p)$ and $J_{ns}^+(p)$ are equal. When the genus of $X_0(p)$ is $0$ (equivalently, when there are no newforms at level $p$, for example when $p=13$), the following lemma gives a concrete way of verifying this.

\begin{lemma}\label{equaljacrank} Suppose the genus of $X_0(p)$ is $0$, and that the rank of the abelian variety $\ker(u_{p^2}+1)$ is zero. Then $J_{ns}(p)$ and $J_{ns}^+(p)$ have equal rank.
\end{lemma}

\begin{proof} By assumption, $J_0(p^2)_{\mathrm{new}} \sim J_0(p^2)$ and $J_0^+(p^2)_{\mathrm{new}} \sim J_0^+(p^2)$. Then using Chen's isogenies (Theorem \ref{Chen}), $J_{ns}(p) \sim J_0(p^2)$ and $J_{ns}^+(p) \sim J_0^+(p^2) \sim J_0(p^2) / \langle u_{p^2} \rangle$. Since $u_{p^2}$ is an involution, we have  \[J_0(p^2) \sim \ker(u_{p^2}+1) \times \ker(u_{p^2}-1) \sim \ker(u_{p^2}+1) \times J_0^+(p^2). \] So \[\mathrm{Rk}(J_{ns}(p))-\mathrm{Rk}(J_{ns}^+(p))=\mathrm{Rk}(J_0(p^2))-\mathrm{Rk}(J_0^+(p^2)) = \mathrm{Rk}(\ker(u_{p^2}+1)). \]  \end{proof} By considering the $L$-function of the abelian variety $\ker(u_{p^2}+1)$ as above, we can try and check whether or not $\mathrm{Rk}(\ker(u_{p^2}+1))=0$.

\subsection{Quadratic Points}

We now turn our attention to quadratic points on $X_{ns}(p)$. Continuing with the same notation, we write $ \varrho : X_{ns}(p) \rightarrow X_{ns}^+(p)$ for the degree $2$ degeneracy map and $w_p$ for the modular involution on $X_{ns}(p)$. We would like to describe all quadratic points on $X_{ns}(p)$. We note \citep[p.~96]{domes} that  $X_{ns}(p)(\R) = \emptyset$, and so  all quadratic points will come in pairs, defined over some imaginary quadratic field. 

One way of obtaining quadratic points on $X_{ns}(p)$ is by pulling back rational points on $X_{ns}^+(p)$: if $P \in X_{ns}^+(p)(\Q)$ then $\varrho^*(P)$ is a pair of quadratic points on $X_{ns}(p)$. We call such quadratic points \emph{non-exceptional}. If on the other hand we have a quadratic point $Q \notin \varrho^* X_{ns}^+(p)(\Q)$, then we say that $Q$ is \emph{exceptional}. We recall here the main theorem of this paper.
\begingroup
\renewcommand\thetheorem{1.1}
\begin{theorem}[Main theorem]
There are no exceptional quadratic points on $X_{ns}(p)$ for $p = 7, 11,$ or $13$.
\end{theorem}
\endgroup

As discussed in the introduction, quadratic points on the curves $X_0(N)$ of small genus are well understood. Due to the similarities between $X_0(p)$ and $X_{ns}(p)$, one might hope to be able to apply similar techniques to those in \citep{ozmansiksek, joshaquad} to study quadratic points on $X_{ns}(p)$. The equations for the curves $X_{ns}(7)$ and $X_{ns}(11)$ are relatively simple and Theorem \ref{mainthm} can be proved for these curves fairly directly. The curve $X_{ns}(13)$ is of genus $8$ and its equations are naturally much more complicated. For this case we use the method of Chabauty for symmetric powers of curves developed in \citep{Symmchab}. The application of this method is similar to how it is used in \citep{joshaquad} for some curves $X_0(N)$. The main difference in the case of $X_{ns}(13)$ is the inability to get a handle on the Mordell--Weil group of its Jacobian. 

\section{Quadratic Points on $X_{ns}(7)$}

The curve $X_{ns}(7)$ is of genus $1$ and the curve $X_{ns}^+(7)$ is of genus $0$. An affine model for $X_{ns}(7)$ is given in \citep[p.~24]{zywina} by the following equation: \[y^2=-(2x^4-14x^3+21x^2+28x+7),\] with the map $\varrho: X_{ns}(7) \rightarrow X_{ns}^+(7) \cong \mathbb{P}^1$ given by $(x,y) \mapsto x$. Homogenising gives the following model in weighted projective space $\mathbb{P}_{(1,2,1)}$: \[ Y^2=-(2X^4-14X^3Z+21X^2Z^2+28XZ^3+7Z^4). \]  We apply the change of coordinates $ (X:Y:Z) \mapsto \left(\frac{3X}{2}:Y:-\frac{3X}{2}+2Z \right) $ to obtain a new model for $X_{ns}(7)$ in $\mathbb{P}_{(1,2,1)}$ with equation \[ Y^2=-81X^4+189X^2Z^2-112Z^4. \] The map $\varrho$ in these new coordinates is given by $\varrho(X:Y:Z)=(3X/2 : -3X/2+2Z).$ Let $P \coloneqq (u:v:w)$ be a quadratic point  on this model. The points at infinity on this model are given by $\infty_+ = (1:9i:0)$ and $\infty_- = (1:-9i:0)$, where $i$ denotes $\sqrt{-1}$. As the pair of points at infinity are non-exceptional, we can assume that $w = 1$. Then it will be enough to show that $u \in \Q$ to show that $P$ is non-exceptional, as if $u \in \Q$, we have $\varrho(P)=(3u/2:3u/2+2) \in \mathbb{P}^1(\Q) \cong X_{ns}^+(7)(\Q)$. 

The reason for using this new model is that it is a double cover of the conic $\mathcal{C} \subseteq \mathbb{P}^2$ given by $Y^2=-81X^2+189XZ-112Z^2$; the covering map being $(X,Y,Z) \mapsto (X^2,Y,Z^2)$. Note that $\mathcal{C}(\Q)=\emptyset$. We will use this to better understand $X_{ns}(7)$. We follow the proof of Siksek \citep{samirnote} to show that $P$ is non-exceptional.

\begin{lemma} The group $\mathrm{Pic}^0(X_{ns}(7)/\Q)=0$.
\end{lemma} 

\begin{proof}  Since $X$ has genus $1$, $J_{ns}(7)$ is an elliptic curve. By Chen's isogeny (Theorem \ref{Chen}), $J_{ns}(7) \sim J_0(49)_{\mathrm{new}}$, so $J_{ns}(7)$ must be an elliptic curve of conductor $49$. There are four elliptic curves of conductor $49$, each of which has Mordell--Weil group $\Z / 2\Z $, so we must have $J_{ns}(7)(\Q) \cong \Z / 2\Z $. In fact one can verify that $J_{ns}(7)$ is the elliptic curve with Cremona reference 49A2.

From the inclusion $\mathrm{Pic}^0(X_{ns}(7)/\Q) \hookrightarrow J_{ns}(7)(\Q) \cong \Z / 2\Z $, we know that $\mathrm{Pic}^0(X_{ns}(7)/ \Q)=0$ or $\Z / 2 \Z$. To prove that $\mathrm{Pic}^0(X_{ns}(7)/ \Q)=0$, we must show that the non-trivial rational divisor class of degree $0$ cannot be represented by a rational divisor. We work on the affine patch $Z=1$. Our model on this affine patch is given by the equation  \[ y^2=-81x^4+189x^2-112. \] The non-trivial rational divisor class of degree $0$ can be represented by the divisor $D \coloneqq\infty_{\scriptscriptstyle+} - \infty_{\scriptscriptstyle-}$, since a straightforward check shows that $\mathrm{div(g)}=2D$, where \[ g \coloneqq 2y+18ix^2+21i. \]  Suppose for a contradiction that $D \sim D'$ with $D'$ a rational degree $0$ divisor. By Riemann-Roch, $D'+\infty_{\scriptscriptstyle +}  + \infty_{\scriptscriptstyle-}$ is lineary equivalent to an effective degree $2$ rational divisor, say \[ D'+\infty_{\scriptscriptstyle +}  + \infty_{\scriptscriptstyle-} \sim Q+Q',\] where $Q$ is a quadratic point and $Q'$ its Galois conjugate. Then \begin{equation}\label{linequiv} Q+Q'-2\infty_{\scriptscriptstyle +} \sim D' -(\infty_{\scriptscriptstyle+}-\infty_{\scriptscriptstyle-}) \sim D'-D \sim 0.\end{equation} Since $\infty_{\scriptscriptstyle+} \in X_{ns}(7)(\Q(i))$, by (\ref{linequiv}) there exists $f \in \Q(i)(X_{ns}(7))$  such that $\mathrm{div}(f)=Q+Q'-2\infty_{\scriptscriptstyle+}$. So $f \in  \mathcal{L}(2\infty_{\scriptscriptstyle+})$ and $f(Q)=f(Q')=0$. The elements $1,y+9ix^2$ form a  $\Q(i)$-basis for $ \mathcal{L}(2\infty_{\scriptscriptstyle+})$, so after rescaling we can write $f=y+9ix^2+\alpha$ for some $\alpha \in \Q(i)$. The vanishing locus of $f$ on the curve $X_{ns}(7)$ is given by \[ x^2= \frac{-112-\alpha^2}{18i\alpha-189}, \qquad y = \frac{112i+i\alpha^2}{18i\alpha-189}-\alpha. \] So $x(Q)^2=x(Q')^2$ and $y(Q)=y(Q')$. As $Q$ and $Q'$ are Galois conjugates, we must have that $x(Q)^2,y(Q) \in \Q$. This means that $(x(Q)^2,y(Q))\in \mathcal{C}(\Q)$, but $\mathcal{C}(\Q) = \emptyset$, a contradiction. 
\end{proof}

Using this Lemma, the proof of Theorem \ref{mainthm} for $p=7$ follows quickly.  On the affine patch $Z=1$, we have $P=(u,v)$, and we write $P' \coloneqq (u',v')$ for the quadratic conjugate of $P$. Then $D \coloneqq P+P'$ is a rational effective degree $2$ divisor, and $\infty_+ + \infty_-$ is too. So $D-\infty_+-\infty_- \sim 0 $ because $\mathrm{Pic}^0(X_{ns}(7)/\Q)=0$. It follows that there is an $h \in \mathcal{L}(\infty_+ + \infty_-)$ satisfying $h(P)=h(P')=0$. Since $X$ has genus $1$, $l(\infty_+ + \infty_-)=2$ and $\{1,x\}$ is a $\Q$-basis for this Riemann-Roch space. This means that we can write, after rescaling, $h=x-\alpha$ for some $\alpha \in \Q$. Then $0=h(P)=u-\alpha$, so $u=\alpha \in \Q$ as required.

\section{Quadratic Points on $X_{ns}(11)$}

The curve $X_{ns}(11)$ is of genus $4$ and a model for this curve is given in \citep[p.~97]{domes} by the following equations in $\mathbb{A}^3_{x,y,t}$:
\begin{align*} y^2+y & =  x^3-x^2-7x+10,\\
t^2 & = -(4x^3+7x^2-6x+19).
\end{align*}
Each of these two equations defines an elliptic curve in its own right. The first of the two equations is in fact the defining equation of $X_{ns}^+(11)$; that is, the curve $X_{ns}^+(11)$ is the elliptic curve defined by
\[ y^2+y=x^3-x^2-7x+10. \]
This curve has rank $1$ and its rational points form an infinite cyclic group generated by the point $(4,-6)$.

The map $\varrho: X_{ns}(11)\rightarrow X_{ns}^+(11)$ is then given by $(x,y,t) \mapsto (x,y)$. This means that to verify Theorem \ref{mainthm} for $p=11$, it will suffice to show that if $(u,v,w) \in X_{ns}(11)$ is a quadratic point, then $u,v \in \mathbb{Q}$; with the slight caveat that we must consider the points at infinity which are not on this affine patch, but these are indeed non-exceptional.

The way we show Theorem \ref{mainthm} holds for $p=11$ is to use the maps from $X_{ns}(11)$ down to various elliptic curves, namely the elliptic curves over $\Q$ with conductor $121$. By Chen's isogeny, we know that $J_{ns}(11) \sim J_0(11^2)_{\mathrm{new}}$. At level $121$ there are four Galois-conjugacy classes of newforms, and to each of these is associated an elliptic curve of conductor $121$, so \[J_{ns}(11) \sim J_0(121)_{\mathrm{new}} \sim A_1 \times A_2 \times A_3 \times A_4,\] the product of these four elliptic curves. Here, $A_2$ is the elliptic curve of rank $1$ isomorphic to $X_{ns}^+(11)$, $A_3$ is the elliptic curve mentioned above with defining equation \[ t^2 = -(4x^3+7x^2-6x+19) \] which has trivial Mordell--Weil group; we denote this curve by $E$. Finally $A_1$ and $A_4$ also have trivial Mordell--Weil group. Moreover, \citep[pp.~100-101]{domes} gives the maps to these elliptic curves. 

\begin{lemma} Let $P \coloneqq (u,v,w) \in X_{ns}(11)$ be a quadratic point. Then $u \in \Q$.
\end{lemma}

\begin{proof}

Write $\Q(P)=\Q(\sqrt{d})$ with $d$ a squarefree integer, and write $P'=(u',v',w')$ for the quadratic conjugate of $P$.
 
To show $u \in \Q$, we consider the map $\psi: X_{ns}(11)\rightarrow E$ given by $(x,y,t) \mapsto (x,t)$. Write $D \coloneqq P + P'$ for the rational effective degree $2$ divisor on $X_{ns}(11)$, and $\psi(D)=(u,w)+(u',w')$ for its pushforward by $\psi$, which is a rational effective degree $2$ divisor on the elliptic curve $E=A_3$. Since $E(\Q)=\{ \infty_{\scriptscriptstyle E} \}$, $\mathrm{Pic}^0(E)=0$, so $\big[\psi(D)-2\infty_{\scriptscriptstyle E}  \big]=[0]$.  So $\psi(D) \sim 2\infty_{\scriptscriptstyle E}$. Now $\psi(D) \neq 2\infty_{\scriptscriptstyle E}$ because $(u,w)$ is an  affine point of the curve, so we have a non-constant function $h \in \Q(E)$ satisfying \[ \mathrm{div}(h) = \psi(D)-2 \infty_{\scriptscriptstyle E}.\] So $h \in \mathcal{L}(2 \infty_{\scriptscriptstyle E})$ with $h(u,w)=h(u',w')=0$. Now $1,x \in  \mathcal{L}(2 \infty_{\scriptscriptstyle E})$ form a $\Q$-basis for this Riemann-Roch space, so we can write (after rescaling) $h=x-\alpha$ for some $\alpha$ in $\Q$. Then $0 = h(u,w) = u - \alpha$, so $u \in \Q$.
\end{proof}

Next we aim to show that $v \in \Q$. Since $E(\Q)=\{ \infty_{\scriptscriptstyle E} \}$, we know that $w$ cannot also be rational because $u \in \Q$, so $w \in \Q(\sqrt{d}) \backslash \Q$. As $\Q(P)=\Q(\sqrt{d})$ we must have $v \in \Q(\sqrt{d})$. Suppose for a contradiction that $v \notin \Q$. As $P \in X_{ns}(11)$, we know that $v$ and $w$ satisfy the following two equations; which we view as quadratic equations in $v$ and $w$ respectively: \[ v^2+v -(u^3-u^2-7u+10) = 0, \qquad  w^2  +(4u^3+7u^2-6u+19) = 0 .\]
Since $v,w \in \Q(\sqrt{d}) \backslash \Q$, the discriminant of each of these equations is of the form $da^2$ for some non-zero rational number $a$; say $da_1^2$ and $da_2^2$ for the first and second equations respectively. In particular, by considering the product of the discriminants, we see that $(da_1a_2,u)$ is a rational point on the curve with affine equation in $\mathbb{A}^2_{r,s}$, \[ r^2 = (4s^3 - 4s^2 - 28s + 41)(-16s^3 - 28s^2 + 24s - 76). \] Replacing $r$ by $2r$, we have an affine rational point on the curve which we denote $H$, with affine equation \[ r^2 = -(4s^3 - 4s^2 - 28s + 41)(4s^3 + 7s^2 - 6s + 19). \]
This curve is a hyperelliptic curve of genus $2$ and already appears in \citep[p.~100]{domes}, since the maps from $X_{ns}(11)$ to both $A_1$ and $A_4$ factor through $H$. To obtain our desired contradiction it will be enough to show $H(\Q) = \emptyset$. Denote by $\pi$ the map (over $\Q$) from $H$ to $A_1$, and write $\infty_{\scriptscriptstyle A_1}$ for the point at infinity on the affine Weiestrass model for $A_1$ given in \citep[p.~99]{doublecover}. Since $A_1(\Q)=\{\infty_{\scriptscriptstyle A_1}\}$,  we have that  $H(\Q) \subseteq \pi^{-1}\{\infty_{\scriptscriptstyle A_1}\}=\{{\infty_+,\infty_-}\}$, where $\infty_+=(4i : 1 : 0)$ and  $\infty_-=(-4i : 1 : 0)$ are the two points at infinity on $H$, neither of which is rational, so we conclude that $H(\Q)= \emptyset$. So $v \in \Q$, meaning that $P \in \varrho^*(X_{ns}(11)(\Q)$. This proves Theorem \ref{mainthm} in the case $p=11$.

\section{Quadratic Points on $X_{ns}(13)$}

In order to prove Theorem \ref{mainthm} in the case $p=13$ we will use the method of Chabauty for symmetric powers of curves. 

We first change how we view quadratic points. Given a smooth projective curve $X$ over $\Q$, we write $X^{(2)}$ for its symmetric square. The reason we consider $X^{(2)}$ is because we can view a pair of quadratic points $(P,\overline{P})$ on $X$ as a \emph{rational} point on $X^{(2)}$, and we can simply write $P \in X^{(2)}(\Q)$. A simple way of representing this point is as a degree $2$ effective rational divisor $D \coloneqq P + \overline{P} \in X^{(2)}(\Q)$. From now on, when we speak of quadratic points, or rational points on $X^{(2)}$, or of degree $2$ effective rational divisors, we really mean the same thing. We can rephrase Theorem \ref{mainthm} as $ X_{ns}^{(2)}(p)(\Q)=\varrho^{-1}(X_{ns}^+(p)(\Q))$ for $p =7,11,$ or $13$.

Moreover, if $X^{(2)}(\Q) \neq \emptyset$, then after fixing a basepoint, say $\infty \in X^{(2)}(\Q)$, we have the Abel--Jacobi map \begin{align*} \iota: X^{(2)}(\Q) & \hooklongrightarrow J(X)(\Q) \\
Q & \longmapsto \big[Q-\infty\big]. \end{align*} This allows us to use the Jacobian of the curve $X$ to study quadratic points.

The strategy we use consists of two main steps: \begin{itemize}

\item Chabauty step: for various primes $p$, try and show that the non-exceptional quadratic points are alone in their mod $p$ residue discs. 
\item Sieving step: sieve for unknown quadratic points, hoping to obtain a contradiction. The information obtained from the Chabauty step is used here.
\end{itemize}

Here, the \emph{mod $p$ residue disc} of a point $Q \in X^{(2)}(\Q)$ consists of points $P \in X^{(2)}(\Q)$ such that $\widetilde{P}=\widetilde{Q}$, where $\sim$ denotes reduction modulo $p$.
For computational reasons we start by finding a new model for $X_{ns}(13)$ for which the modular involution is diagonalised (in a sense which is made precise below). We then describe the Chabauty and sieving steps, and see how they are used to complete the proof of Theorem \ref{mainthm}. Finally we see how a saturation test is used to verify certain conditions required in the sieving step.

\subsection{Obtaining a New Model}

As a starting point we use the equations for $X_{ns}(13), X_{ns}^+(13)$, and $\varrho$ from \citep{doublecover}. We first find equations for the modular involution, $w_{13}$, and then use this to obtain a new model. A model for $X_{ns}(13)$ is given by $15$ degree $2$ equations which cut out the curve in $\mathbb{P}^7$. It is a curve of genus $8$. The equations are first obtained via the canonical embedding and then simplified to obtain small coefficients using an LLL algorithm. This model is smooth at all primes $2 \leq p \leq 97$, other than $p=13$. A model for the genus $3$ curve $X^+_{ns}(13)$ is given by the following degree $4$ equation in $\mathbb{P}^2$: \[
 (-Y-Z)X^3+(2Y^2+ZY)X^2  +(-Y^3+ZY^2-2(Z^2)Y+Z^3)X+(2Z^2Y^2-3Z^3Y) = 0. 
\] Finally,  \begin{align*}\varrho: X_{ns}(13) & \longrightarrow X_{ns}^+(13) \\ 
(x_1: \dots :x_8) & \longmapsto (-3x_1+2x_2 :-3x_1+x_2+2x_4-2x_5:x_1+x_2+x_4-x_5).
\end{align*} 

The curve $X^+_{ns}(13)$ is now known to have precisely seven rational points, thanks to the important paper \citep{cursed}, which uses the `quadratic Chabauty' method. By pulling back these points, we obtain seven pairs of non-exceptional quadratic points on $X_{ns}(13)$.

Baran finds an explicit form for the $j$-map on $X_{ns}^+(13)$ in \citep[pp.~295-300]{Baran13}. Each of the seven points is a CM-point. The seven rational points, which we denote $P_1, \dots, P_7$, along with the CM fields of the corresponding elliptic curves  are displayed in Table 1. The CM fields of the rational points are precisely the quadratic fields over which their pullbacks are defined.
\begingroup
\renewcommand*{\arraystretch}{1.35}
\begin{table}[ht!]\label{Tab1}
\begin{center}
\begin{tabular}{ |c|c|c| } 
\hline
$P$ & Coordinates & CM \\
\hline
$P_1$ & $(0:1:0)$ & $-11$ \\ 
\hline
$P_2$ & $(0:0:1)$ & $-67$ \\ 
\hline
$P_3$ & $(-1:0:1)$ & $-7$ \\ 
\hline
$P_4$ & $(1:0:0)$ & $-2$ \\ 
\hline
$P_5$ & $(1:1:0)$ & $-19$ \\ 
\hline
$P_6$ & $(0:3:2)$ & $-163$ \\ 
\hline
$P_7$ & $(1:0:1)$ & $-7$ \\ 
\hline
\end{tabular}
\end{center}
\caption{Rational points on $X_{ns}^+(13)$}
\vspace{-5mm}
\end{table} \endgroup \kern-0.9em

The modular involution interchanges each pair of quadratic points, and so we know the images of these fourteen points under $w_{13}$. Using some linear algebra we obtain an involution on the curve defined over $\mathbb{Q}$, and since the modular involution is the unique involution on the curve, we have \[ w_{13}  (x_1 : \dots : x_8) =  (x_1 : x_2 : x_1-x_2-x_3 : -x_5 : -x_4 : x_1-x_6 : x_2-x_7 : x_1-x_8). \]

We would like to change coordinates so that our modular involution is in a nicer form. As mentioned above, this is purely for computational purposes. To do this, we diagonalise the matrix associated to the modular involution and change coordinates appropriately. The matrix of $w_{13}$ has eigenvalues $1$ and $-1$, with multiplicities $3$ and $5$ respectively. Let $T$ be the change of basis matrix satisfying $Tw_{13}T^{-1} = M$, where $M \coloneqq \mathrm{Diag}(1,1,1,-1,-1,-1,-1,-1)$. We then change coordinates with the matrix $T^{-1}$. The coordinate change induced by the matrix $T^{-1}$ is  \[ (x_1 : \dots : x_8)  \mapsto  \left( x_2 : x_3 : \frac{x_2-x_3-x_5}{2} : \frac{x_1+x_8}{2}: \frac{-x_1+x_8}{2}:  \frac{x_2-x_6}{2} : \frac{x_3-x_7}{2} : \frac{x_2-x_4}{2} \right). \]  
As we can see from the denominators appearing in the coordinate change, this introduces $2$ as a prime of bad reduction on our new model. This will not be an issue. Applying this coordinate change to the equations of our curve and to the map $\varphi$ gives us our new model and new map to $X_{ns}^+(13)$, with the model for $X_{ns}^+(13)$ left unchanged. These equations are given in the Appendix.

We now pull back the seven rational points on $X_{ns}^+(13)$ to obtain seven pairs of quadratic points on $X_{ns}(13)$. We denote these by $Q_1, \dots, Q_7$, and their coordinates are listed in Table 2. Our aim is to show that this list is complete.
\begingroup
\renewcommand*{\arraystretch}{2}
\begin{table}[ht!]
\begin{center}
\begin{tabular}{ |c|c|c|c|c| } 
 \hline
 $Q_i$ & $\theta^2$ & Coordinates & CM  & $\varrho(Q_i)=P_i$ \\
 \hline 
$Q_1$ & $-11$ &  $\left( -\frac{5 \theta}{13} : \frac{2 \theta}{13} : \frac{3 \theta}{13} : 0 : -1 : -2 : 1 : 1 \right)$ & $-11$ & $(0:1:0) $ \\ 
\hline
$Q_2$ & $-67$ & $\left(\frac{3 \theta}{13} :\frac{4 \theta}{13} : \frac{6 \theta}{13} : 0 : 4 : -4 : -2 : 1\right)$ & $-67$ & $(0:0:1)$ \\ 
\hline
$Q_3$ & $-7$ & $\left(\frac{7 \theta}{13} : 
    \frac{5 \theta}{13} : \frac{ \theta}{13} : -1 : 0 : -1 : 1 : 1     \right)$ & $-7$ & $(-1:0:1)$ \\
\hline
$Q_4$ & $-2$ & $\left( \frac{4 \theta}{13} : \frac{ \theta}{13} : -\frac{5 \theta}{13} : 0 : 0 : 1 : 0 : 0    \right)$ & $-2$ & $(1:0:0)$ \\
\hline
$Q_5$ & $-19$ & $\left(  \frac{ \theta}{13} : -\frac{3 \theta}{13}
: \frac{2 \theta}{13} : 1 : 1 : 1 : 0 : 1   \right)$ & $-19$ & $(1:1:0)$ \\
\hline
$Q_6$ & $-163$ & $\left(  \frac{3 \theta}{13}: \frac{2 \theta}{91}: \frac{3 \theta}{91} : -\frac{12}{7} : -\frac{5}{7} : -\frac{10}{7} : \frac{25}{7} : 1   \right)$ & $-163$ & $(0:3:2)$ \\
\hline
$Q_7$ & $-7$ & $\left(   -\frac{ \theta}{13} : \frac{3 \theta}{13} : \frac{11 \theta}{13} : 1 : 0 : -3 : -1 : 1  \right)$ & $-7$ & $(1:0:1)$ \\
\hline
\end{tabular}
\end{center}
\caption{Quadratic points on $X_{ns}(13)$}
\vspace{-5mm}
\end{table}
\endgroup

We also record here some information about the Jacobians, $J_{ns}(13)$ and $J_{ns}^+(13)$, of these curves. By \citep[Proposition~6.2]{cursed} we know that $J_{ns}^+(13)(\Q)$ has rank $3$. Moreover, by \citep[p.~60]{genus3},  $J_{ns}^+(13)(\Q)$ has trivial torsion, and the divisors  \[ \Delta_i \coloneqq \left[ P_i-P_4 \right], \qquad i=1,2,3,\] generate a full rank subgroup, which we denote $G^+_{13} \subseteq J_{ns}^+(13)(\Q)$. Writing $u_{169}$ for the Atkin--Lehner involution on $X_0(169)$, we verify that the rank of $\ker(u_{169}+1)$ is $0$, so $J_{ns}(13)(\Q)$ also has rank $3$ by Lemma \ref{equaljacrank}, as there are no newforms at level $13$. Therefore, pulling back the divisors $\Delta_i$ to $J_{ns}(13)(\Q)$ will give rise to a  full rank subgroup, $G_{13}$, of this Jacobian. We write $D_i \coloneqq \varrho^*(\Delta_i)$, so that $G_{13}= \langle D_1,D_2,D_3 \rangle$. We write $I$ for the index of $G_{13}$ in $J_{ns}(13)(\Q)$, and $I^+$ for the index of $G^+_{13}$ in $J^+_{ns}(13)(\Q)$.

\subsection{Chabauty Step}

Working more generally, we let $X$ be a smooth projective curve over $\Q$ of genus $g_{\scriptscriptstyle X} \geq 2$, with Jacobian $J(X)$ of rank $r_{\scriptscriptstyle X}$. We suppose that $X$ has an involution, $w$, defined over $\Q$, and define $C \coloneqq X / \langle w \rangle$. Write $\varrho:X \rightarrow C$ for the quotient map. A quadratic point on $X$ is \emph{exceptional} if it does not arise as a pullback of a rational point on $C$. We write $g_{\scriptscriptstyle C}$ for the genus of $C$, and $r_{\scriptscriptstyle C}$ for the rank of the Jacobian of $C$. Let $p$ be a prime of good reduction for both $X$ and $C$. The Chabauty step very much relies on the same pairing used in the usual Chabauty--Coleman method:
\begin{align*} \Omega_{X/\Q_p} \times J(X)(\Q_p)  & \longrightarrow \Q_p  \\ \Big( v, \Big[ \sum_i (P_i-Q_i) \Big] \Big) & \longmapsto \sum_i \int_{P_i}^{Q_i} v,
\end{align*} where $\Omega_{X / \Q_p}$ is the space of regular differentials on the curve $X$ viewed as a curve over $\Q_p$. This pairing is $\Q_p$-linear on the left, $\mathbb{Z}$-linear on the right, and its kernel on the right is $J(X)(\Q_p)_{\mathrm{tors}}$. The integrals are evaluated by expanding $\omega$ as a power series in a uniformiser. In the usual Chabauty--Coleman method, one looks for an \emph{annihilating differential}: a differential $v \in \Omega_{X / \Q_p}$ satisfying  \[ \langle v, D \rangle = 0 \quad \text{for all} \quad  D \in J(X)(\Q_p). \] For the symmetric version, we look for differentials that are both annihilating \emph{and} satisfy $\mathrm{Tr}(v)=0$, where $\mathrm{Tr}$ is the trace operator: $\mathrm{Tr}(v) \coloneqq  v+w^*v$. We write $V_0$ for the space of differentials that are both annihilating and have zero trace, and $\widetilde{V_0}$ for its reduction modulo $p$. 

\begin{theorem}[Symmetric Chabauty, Siksek]\label{chabcorol}  Let $Q = Q_a + Q_b \in \varrho^*(C(\Q)) \subseteq X^{(2)}(\Q)$ be a non-exceptional quadratic point. Let $p$ be a prime of good reduction for $X$ and $C$, and moreover suppose either that $p>3$, or that $p=3$ and  $\widetilde{Q_a} \neq \widetilde{Q_b} \mathrm{~~mod~} 3$. Let $t_{\widetilde{Q_a}}$ be a uniformiser at $\widetilde{Q_a}$. Let $v_1, \dots, v_k$ be a basis for $\widetilde{V_0}$. If, for some $i \in \{1, \dots, k \}$, \begin{equation}\label{Redchabcond} \frac{v_i}{\mathrm{d} t_{\widetilde{Q_a}}}~\Big|_{t_{\widetilde{Q_a}}=0} ~\neq 0, \end{equation} then $Q$ has no exceptional points in its mod $p$ residue disc.
\end{theorem}

We note that $\frac{v_i}{\mathrm{d} t_{\widetilde{Q_a}}}~\Big|_{t_{\widetilde{Q_a}}=0}$ is simply the constant term in the expansion of $v_i$ as a power series in $t_{\widetilde{Q_a}}$. Although not a necessary condition, the rank-genus condition $r_{\scriptscriptstyle X} - r_{\scriptscriptstyle C}   < g_{\scriptscriptstyle X} - g_{\scriptscriptstyle C}$ (which is analogous to the rank-genus condition in the Chabauty--Coleman method) guarantees the existence of an annihilating differential in the kernel of the trace map \citep[Lemma 4.2]{Symmchab}. 

\begin{proof}
For $p >3$ this is the same statement as in \citep[Theorem 2.4]{joshaquad}, which in turn is a variant of the result in \citep[Theorem 4.3]{Symmchab}, since using the notation of \citep[Theorem 4.3]{Symmchab} we have $e=1$ and $N' \leq 2$, so that $\mathrm{ord}_p(i+1) < i/N'$ for all $i > 0$ and $p >3$. The statement here for $p=3$ is slightly sharper, but follows the same ideas. 

Suppose $p=3$ and  $\widetilde{Q_a} \neq \widetilde{Q_b} \mathrm{~~mod~} 3$ as in the theorem. Suppose $P \coloneqq P_a+P_b$ is an exceptional point in the mod $3$ residue disc of $Q$,  and after rearranging if necessary, assume $\widetilde{P_a}=\widetilde{Q_a}$ and $\widetilde{P_b}=\widetilde{Q_b}$. Write  \[ K \coloneqq \Q(Q_a), \quad M \coloneqq \Q(P_a), \quad L \coloneqq \Q(Q_a,P_a) = K \cdot M.\] Then following the proof of \citep[Theorem 4.3]{Symmchab}, it is enough to show that $3$ does not ramify in $L$. Suppose for a contradiction that $3$ ramifies in $L$. Then $3$ must ramify in (at least) one of $K$ and $M$. Suppose $3$ ramifies in $K$, and write $\mathfrak{p}$ for the prime of $K$ above $3$. Then the inertia group of $\mathfrak{p}$ is the full Galois group $\mathrm{Gal}(K/\Q)$. So for $\sigma \in \mathrm{Gal}(K/\Q)$, we have that $Q_b = Q_a^\sigma = Q_a \pmod{\mathfrak{p}}$, a contradiction. If $3$ ramifies in $M$, then $\widetilde{Q_a}=\widetilde{P_a} = \widetilde{P_b}=\widetilde{Q_b} \mathrm{~~mod~} 3,$ another contradiction. 
\end{proof}

In order to compute $\widetilde{V_0}$, we use the following result, based on \citep[pp.~9-10]{joshaquad}.

\begin{proposition}\label{diffcalc}
Assume $r_{\scriptscriptstyle X}=r_{\scriptscriptstyle C}$. Let $p$ be a prime of good reduction for both $X$ and $C$. Denote reduction mod $p$ by $\sim$. Then $ \widetilde{\Omega_{\scriptscriptstyle X}} = \Omega_{\scriptscriptstyle \widetilde{X}}$, and 
\[\widetilde{V_0} = (1-\widetilde{w}^*)(\Omega_{\scriptscriptstyle \widetilde{X}}).\]
\end{proposition}

\begin{proof}
We first show that  $V_0 = (1-w^*)(\Omega_{\scriptscriptstyle X})$. Since $w$ is an involution, $(1-w^*)(\Omega_{\scriptscriptstyle X}) = \ker(1+w^*)$. So $v \in (1-w^*)(\Omega_{\scriptscriptstyle X})$ if and only if $\mathrm{Tr}(v)=0$. So $V_0 \subseteq (1-w^*)(\Omega_{\scriptscriptstyle X})$.

Next, as $r_{\scriptscriptstyle X}=r_{\scriptscriptstyle C}$, we have that $\varrho^*(J(C)(\Q))$ is a full rank subgroup of $J(X)(\Q)$. Write $N$ for its index. Let $D \in J(X)(\Q)$. Then $ND \in \varrho^*(J(C)(\Q)$ and we can write $ND = \varrho^*\Delta$ for some $\Delta \in J(C)(\Q)$. Then for $v \in (1-w^*)(\Omega_{\scriptscriptstyle X})$, using the properties of Coleman integration, \begin{equation*}
\int_0^D v = \frac{1}{N} \int_0^{\varrho^*\Delta}  v = \frac{1}{N} \int_0^D \mathrm{Tr}(v) = 0.
\end{equation*} So $v$ annihilates $J(X)(\Q)$. So $ (1-w^*)(\Omega_{\scriptscriptstyle X}) \subseteq V_0$.

We then effectively reduce everything modulo $p$, as described in \citep[pp.~9-10]{joshaquad}. \end{proof}

\subsection{Sieving Step}

Let $X$ and $C$ be as in Section 5.2. We present here an adaptation of the sieve used in \citep[pp.~226-228]{Symmchab} and apply it in the case $X=X_{ns}(13)$ in Section 5.4. For a more general introduction to this theory, and in particular its application to hyperelliptic curves, we refer to \citep{exposieve}. We start with a non-empty set $\mathcal{L} \subseteq X^{(2)}(\Q)$ of known non-exceptional quadratic points, and a hypothetical unknown exceptional quadratic point $P \in X^{(2)}(\Q)$. In order to apply the sieve we need the following data:

\begin{itemize} 
\item A finite index subgroup $G \subseteq J(\Q)$ with generators $D_1, \dots, D_n$. We write $I \coloneqq [J(\Q):G]$.

\item A non-negative integer parameter, $M$.

\item Primes $p_1, \dots, p_k$ of good reduction for $X$, so that for each $i \in \{1, \dots, k\}$ the index $I$ is coprime to $\#(J(\mathbb{F}_{p_i})/ M J(\mathbb{F}_{p_i}))$.
\end{itemize}

If the index $I$ is known, or even if we know some integer $\hat{I}$ satisfying $\hat{I} \cdot J(\Q) \subseteq G$, then we can adapt the sieve and remove the coprimality assumption. This version of the sieve is described in \citep[pp.~6-7]{joshaquad}. However, usually, and in particular for $X_{ns}(13)$, we will not know the index, or such an integer $\hat{I}$. In this case we use the $p$-saturation method described in Section 5.5 to determine primes that do not divide the index. 

The coprimality condition allows us to  choose an integer, $I^*$, so that for each $i \in \{1, \dots, k\}$, \[I^*I \equiv 1 \mathrm{~~~mod~~} \#\big(J(\mathbb{F}_{p_i})/ M J(\mathbb{F}_{p_i})\big).\] We write $\mu_{p_i,M}:J(\mathbb{F}_{p_i}) \rightarrow J(\mathbb{F}_{p_i})/ M J(\mathbb{F}_{p_i})$ for the natural quotient map. A discussion of the heuristics behind the choice of primes and the parameter $M$ can be found in \citep{exposieve}. 

For the moment, we fix a prime $p \coloneqq p_i$ for some $i$. Choose a base point, which as before we will denote $\infty \in X^{2}(\Q)$, and write $\iota:  X^{2}(\Q) \hookrightarrow J(\Q)$ for the corresponding Abel--Jacobi map. Define \begin{align*} \varphi:   \mathbb{Z}^n & \longrightarrow G \subseteq J(\Q) \\  (a_1, \dots, a_n)& \longmapsto a_1D_1+ \dots  + a_nD_n. \end{align*} We write $\varphi_{p,M}$ for the map obtained by first applying $\varphi$, then reducing modulo $p$, and then applying $\mu_{p,M}$.  Denote reduction modulo $p$ by $\sim$. The map $\iota_p$ is the Abel--Jacobi map on $X^{(2)}(\mathbb{F}_p)$ with basepoint $\widetilde{\infty}$. Finally $\iota_{p,M}$ is the composition of $\iota_p$ with the quotient map $\mu_{p,M}$. We obtain the following commutative diagram.
\begin{equation}\label{SieveDiagram}  \begin{tikzcd}[sep = huge] 
 \mathcal{L}  \arrow[hook]{r} & [-25pt] X^{(2)}(\Q) \arrow[hook]{r}{\iota}  \arrow[swap]{d}{\sim} & J(\Q)  \arrow{d}{\sim} & [-30pt] G \arrow[hook']{l} & \mathbb{Z}^n \arrow[two heads,swap]{l}{\varphi} \arrow[two heads]{ddll}{\varphi_{p,M}} 
\\ 
& X^{(2)}(\mathbb{F}_p) \arrow[hook]{r}{\iota_p} \arrow[swap]{dr}{\iota_{p,M}}  &  J(\mathbb{F}_p) \arrow[two heads]{d}{\mu_{p,M}}
\\ & &  \frac{J(\mathbb{F}_p)}{ M J(\mathbb{F}_p)}     
\end{tikzcd} \end{equation} \normalsize

Consider $P \in X^{(2)}(\Q) \backslash \mathcal{L}$, our (hypothetical) unknown quadratic point. We see that $I \cdot \iota (P) \in G$, so we can write \[I \cdot \iota (P)=a_1D_1 + \dots + a_nD_n = \varphi(a_1, \dots, a_n),\] for some $(a_1, \dots, a_n) \in \Z^n$. Multiplying through by $I^*$, we have \[(I^*I) \cdot  \iota (P)=I^*a_1D_1 + \dots + I^*a_nD_n = \varphi(I^*a_1, \dots, I^*a_n).\]
Reducing this expression modulo $p$ and applying the quotient map $\mu_{p,M}$ gives, in the group $J(\mathbb{F}_{p_i})/ M J(\mathbb{F}_{p_i})$,
\[   \mu_{p,M} \big((I^*I) \cdot \widetilde{\iota(P)} \big) = I^*a_1\mu_{p,M}(\widetilde{D_1}) + \dots + I^*a_n\mu_{p,M}(\widetilde{D_n}).\] Since Diagram \ref{SieveDiagram} commutes,  $\mu_{p,M} \big((I^*I) \cdot \widetilde{\iota(P)} \big)= (I^*I) \cdot \iota_{p,M}(\widetilde{P})$. Also, by our choice of $I^*$, we have  $(I^*I) \cdot \iota_{p,M}(\widetilde{P})=\iota_{p,M}(\widetilde{P})$ in $J(\mathbb{F}_{p_i})/ M J(\mathbb{F}_{p_i})$. So \[\iota_{p,M}(\widetilde{P})= I^*a_1\mu_{p,M}(\widetilde{D_1}) + \dots + I^*a_n\mu_{p,M}(\widetilde{D_n})=\varphi_{p,M}(I^*a_1, \dots, I^*a_n). \] We conclude that $\iota_{p,M}(\widetilde{P}) \in \varphi_{p,M}(\Z^n)=\mu_{p,M}(\widetilde{G})$. 

From this, we obtain a set of points in $X^{(2)}(\mathbb{F}_p)$ that P could apriori reduce to modulo $p$, namely \[\mathcal{S}_{p,M} \coloneqq \{ D \in X^{(2)}(\mathbb{F}_p) : \iota_{p,M}(D) \in \mu_{p,M}( \widetilde{G}) \}. \] 
We now use the information obtained from the Chabauty step to eliminate as many of the points in $\mathcal{S}_{p,M}$ as possible. Write $\mathcal{H}_{p,M} \subseteq X^{(2)}(\mathbb{F}_p)$ for the reductions of points that satisfy the test in Theorem \ref{chabcorol}. Then $\widetilde{P} \neq \widetilde{Q}$ for any $\widetilde{Q} \in \mathcal{H}_{p,M}$ as otherwise $P$ would be an exceptional point in the mod $p$ residue disc of $Q$, contradicting $Q \in \mathcal{H}_{p,M}$. This leaves us with a set $ \mathcal{T}_{p,M} \coloneqq \mathcal{S}_{p,M} \backslash \mathcal{H}_{p,M} \subseteq X^{(2)}(\mathbb{F}_p)$ of possibilities for $\widetilde{P}$. The more points that pass the Chabauty test the better, as this means we are eliminating more possibilities for $ \widetilde{P}$. We can calculate $\mathcal{T}_{p,M}$ explicitly.

Since $ \iota_{p,M}(\mathcal{T}_{p,M}) \subseteq \mu_{p,M}(\widetilde{G}) = \varphi(\mathbb{Z}^m)$, we obtain a set, $\mathcal{W}_{p,M}$, of $\mathcal{B}_{p,M} \coloneqq \ker(\varphi_{p,M})$ cosets: \[\mathcal{W}_{p,M} \coloneqq \varphi_{p,M}^{-1}( \iota_{p,M}(\mathcal{T}_{p,M})). \] 

This set of cosets, $\mathcal{W}_{p,M}$, encodes the list of possibilities for $ I \cdot \iota(P)$; namely if $I \cdot \iota(P) = a_1D_1+ \dots + a_nD_n$, then $(I^*a_1, \dots, I^*a_n) \in u+\mathcal{B}_{p,M}$ for some $u+\mathcal{B}_{p,M} \in \mathcal{W}_{p,M}$.

Although $\mathcal{W}_{p,M}$ is obtained by investigating matters modulo $p$, the information it encodes is completely independent of $p$. This means that if we choose a different prime of good reduction, say $p'$, then just as above, if $I \cdot \iota(P) = a_1D_1+ \dots+ a_nD_n$, then $(I^*a_1, \dots, I^*a_n) \in v+\mathcal{B}_{p',M}$ for some $v+\mathcal{B}_{p',M} \in \mathcal{W}_{p',M}$. So $(I^*a_1, \dots, I^*a_n) \in \mathcal{W}_{p,p',M} \coloneqq \mathcal{W}_{p,M} \cap \mathcal{W}_{p',M}$.

We then repeat this process for each prime $p$ in our list $p_1, \dots,p_k$ of primes of good reduction. We hope that $\mathcal{W}_{p_1, \dots, p_k,M} \coloneqq \cap_{i=1}^k \mathcal{W}_{p_i,M} = \emptyset$, as if this is the case, it follows that $I \cdot \iota(P)$ is not expressible as a linear combination of the $D_i$, and thus $I \cdot \iota(P) \notin G$, meaning that $P \notin X^{2}(\Q)$, giving us our desired contradiction. We state this sieving principle as a proposition.

\begin{proposition}[Sieving Principle]
Let $p_1, \dots, p_k$ be primes of good reduction for $X$, and let $M$ be a non-negative integer. Suppose \[\mathcal{W}_{p_1,\dots,p_k,M}  = \emptyset.\] Then $X^{(2)}(\Q)$ contains no exceptional quadratic points.

\end{proposition}

We note that the integer $M$ plays two roles. The first is crucial. We require $I$ to be coprime to  $\#(J(\mathbb{F}_{p_i})/ M J(\mathbb{F}_{p_i}))$ for each $i$. If there is a prime $p \mid \#J(\mathbb{F}_{p_i})$ which is either not coprime to the index, or we cannot show  it is coprime to the index, then by not including this prime as a factor of $M$, we have that $p \nmid \#(J(\mathbb{F}_{p_i})/ M J(\mathbb{F}_{p_i}))$. This allows us to remove any troublesome primes. The downside to this is that we potentially lose information in the sieve.  We also note that if the index, or an integer divisible by the index, is not known then we cannot set $M = 1$.

 Secondly, $M$ allows us to remove (in the quotient) any large primes appearing in the factorisation of $\#J(\mathbb{F}_{p_i})$. This can help avoid a combinatorial explosion due to the Chinese Remainder Theorem. Indeed, as we intersect the subgroups $\mathcal{B}_{p,M}$, the resulting subgroup can have very large index in $\Z^n$, which slows down calculations.

\subsection{Applying the Sieve }

We now exhibit some of the sieving step calculations. We set $X$ to be the curve $X_{ns}(13)$, $\mathcal{L} \coloneqq \{ Q_1, \dots, Q_7 \}$, and $G \coloneqq G_{13}$.

We first make some remarks on the Chabauty step. The curves $X_{ns}(13)$ and $X_{ns}^+(13)$ have equal rank, and so we can apply Proposition \ref{diffcalc} to compute $\widetilde{V_0}$. For each known quadratic point $Q_i$, we verify that $\widetilde{Q_{i,a}} \neq \widetilde{Q_{i,b}} \mathrm{~~mod~} 3$, meaning that the conditions of Theorem \ref{chabcorol} hold for all primes $p \geq 3$ other than $13$ (which is a prime of bad reduction). We found that each point $Q_i$ is alone in its mod $p$ residue disc for all primes $3 \leq p \leq 73$ (other than $13$). This is the best possible result we could hope for, as it greatly increases chances of success in the sieving step.

As we will see in Section 5.5, the integers $3,5,13,$ and $29$ are coprime to $I=[J_{ns}(13)(\Q) : G_{13}]$. This means that we can choose to set $M=3^{10} \cdot 5^{10}\cdot 13^{10} \cdot 29^{10} $ (here, the exponent $10$ is simply large enough to ensure we keep the corresponding factors in the quotient).

We start by applying the sieve with $p=3$. Our aim is to calculate $\mathcal{W}_{3,M}$. To do this, we must first calculate $\mathcal{T}_{3,M}$. We find that $X_{ns}(13)^{(2)}(\mathbb{F}_3)$ consists of $27$ points. By considering $\iota_{3,M}(R)$ for each $R \in X_{ns}(13)^{(2)}(\mathbb{F}_3)$, and seeing which of these lie in $\mu_{3,M}(\widetilde{G_{13}})$, we obtain the set $\mathcal{S}_{3,M}$ which consists of nine points. As each of our seven points satisfies the Chabauty criterion, \[\mathcal{H}_{3,M}= \big\lbrace\widetilde{Q_i} : i = 1, \dots, 7 \big\rbrace.\] Since $\widetilde{Q_2} = \widetilde{Q_6}$, $\# \mathcal{H}_{3,M}=6$, and removing these six points from $\mathcal{S}_{3,M}$ gives \begin{align*}
\mathcal{T}_{3,M} = \big\lbrace & (0 : 0 : 0 : t^5 : t^5 : 2 : t^2 : 1)  + (0 : 0 : 0 : t^3 : t^3 : 2 : t^6 : 1), \\
   &   (0 : 0 : 0 : t^3 : t^7 : t^7 : t^5 : 1)  + (0 : 0 : 0 : t^5 : t : t : t3 : 1), \\ 
    &  (0 : 0 : 0 : t^5 : t^6 : t^5 : 1 : 0)  + (0 : 0 : 0 : t^3 : t^2 : t^3 : 1 : 0) \big\rbrace, 
\end{align*}
where $t$ satisfies $t^2-t-1 = 0$. Applying $\iota_{3,M}$ to this set, we obtain \[\iota_{3,M}(\mathcal{T}_{3,M})= \big\lbrace (8,8),(7,7),(9,9) \big\rbrace \in \frac{J_{ns}(13)(\mathbb{F}_3)}{MJ_{ns}(13)(\mathbb{F}_3)} \cong \Z/13\Z \oplus \Z / 13\Z .\] Writing $\mathbb{Z}^3 \coloneqq \langle e_1,e_2,e_3 \rangle$ with $e_i$ the standard basis vectors, we find that \begin{align*} \mathcal{B}_{3,M} & =  \ker(\varphi_{3,M}) = \langle e_1+7e_3, e_2 + 5e_3, 13e_3 \rangle, \\ \mathcal{W}_{3,M} & = \{-3e_1+\mathcal{B}_{3,M}, 4e_1+\mathcal{B}_{3,M}, -2e_1+\mathcal{B}_{3,M} \}.\end{align*} The index of $B_{3,M}$ in $\Z^3$ is $13$.  This completes the calculations for $p =3$.

We now repeat these calculations for $p=5$. This provides us with a group $\mathcal{B}_{5,M}$ and a set of $\mathcal{B}_{5,M}$ cosets, $\mathcal{W}_{5,M}$. We find that  \begin{align*} \mathcal{B}_{5,M} & \coloneqq \langle e_1+6e_3, e_2 + 320e_3, 377e_3 \rangle, \\ \mathcal{W}_{5,M} & \coloneqq \{-57e_1+\mathcal{B}_{5,M}, -165e_1+\mathcal{B}_{5,M}, 28e_1+\mathcal{B}_{5,M} \}. \end{align*}
Although $\mathcal{W}_{3,M}$ and $\mathcal{W}_{5,M}$ were obtained from information modulo the corresponding prime, the information encoded is independent of the prime used. We now intersect $\mathcal{W}_{3,M}$ and $\mathcal{W}_{5,M}$ to find a new list of cosets, which will now be $\mathcal{B}_{3,5,M} \coloneqq \mathcal{B}_{3,M}\cap \mathcal{B}_{5,M}$ cosets. We obtain 
\begin{align*}  \mathcal{B}_{3,5,M} & \coloneqq  \langle e_1+9e_2+247e_3, 13e_2+13e_3, 377e_3 \rangle, \\
 \mathcal{W}_{3,5,M} & \coloneqq \{-165e_1 + \mathcal{B}_{3,5,M}, ~ 274e_1 -5e_3 + \mathcal{B}_{3,5,M}, ~ -371e_1 + e_3 + \mathcal{B}_{3,5,M}    , \\ &    \qquad 445e_1 - 4e_3 + \mathcal{B}_{3,5,M}, ~   -1187e_1 +6e_3 + \mathcal{B}_{3,5,M}, \\ & \qquad 29e_1 + 6e_3 + \mathcal{B}_{3,5,M}, ~ -475e_1-2e_3+\mathcal{B}_{3,5,M},  \\ & \qquad  91e_1 + e_3 + \mathcal{B}_{3,5,M},~ -981e_1 + 5e_3 + \mathcal{B}_{3,5,M} \}. \end{align*}
Our list of cosets has grown, and so it may seem like we are worse off than we were before. However, these cosets are cosets of a much smaller subgroup: the index of $\mathcal{B}_{3,5,M}$ in $\Z^3$ is $4901$. By continuing in this manner, we eventually find that \[\mathcal{W}_{3,5,31,43,53,61,73, M} = \emptyset,\] proving that there are no exceptional quadratic points on $X_{ns}(13)$. Table 3 shows how the size of $\mathcal{W}_{p_1, \dots, p_k,M}$ and the index of $\mathcal{B}_{p_1, \dots, p_k,M }$ in $\mathbb{Z}^3$ varies as we increase the number of primes.
 
\begingroup
\footnotesize
\renewcommand*{\arraystretch}{1.5}
\begin{table}[ht!]\label{Tab3}
\begin{center}
\begin{tabular}{ |c|c|c|c|c|c|c|c| } 
\hline
$p_1, \dots, p_k$ & $3$ & $3,5$ & $3,5,31$ & $3,5,31,43$ & $3,5,31,43,53$  & $3,5,31,43,53,61$  & $3,5,31,43,53,61,73$  \\ \hline 
$ [ \mathbb{Z}^3 : \mathcal{B}_{p_1, \dots, p_k,M} ] $ & $13$ & $4901$ & $142129$ & $1847677$ & $1847677$ & $1847677$ & $1847677$ \\ \hline
$ \# \mathcal{W}_{p_1, \dots, p_k,M}$ & $3$ & $9$ & $27$ & $351$ & $247$ & $65$ & $0$ \\ \hline
\end{tabular}
\end{center}
\caption{Sieving}
\vspace{-6mm}
\end{table} \endgroup  

For our choice of $M$, we first chose all primes $p \leq 73$ for which $J(\mathbb{F}_p) / M J(\mathbb{F}_p) \neq 0$ and  $ [\mathbb{Z}^3 : \mathcal{B}_{p,M} ] >1$. We then applied the sieve, and removed any primes which did not alter the size of $\mathcal{W}_{p_1, \dots, p_k,M}$, leaving the primes $3,5,31,43,53,61$, and $73$. We note that a different choice of $M$, and a different choice of primes, or using the same primes in a different order, may reduce the computation time. Due to computational restraints, it is impractical to use primes $>100$ in the sieving process.

\subsection{Testing Saturation}

It remains to show that the index, $I$, of $G_{13}$ in $J_{ns}(13)(\Q)$, is not divisible by the primes $3,5,13,$ and $29$. The following result allows us to work with the subgroup $G_{13}^+$ of $J_{ns}^+(13)(\Q)$. This step is important to avoid the computational issues of working directly with $J_{ns}(13)$.

\begin{proposition}\label{index} Writing $I=[J_{ns}(13)(\Q) : G_{13}]$ and $I^+=[J^+_{ns}(13)(\Q) : G^+_{13}]$ as before, we have $I \mid 14I^+$. In particular, if $l$ is a prime with $l \nmid I^+$ and $l \neq 2,7$, then $l \nmid I$.
\end{proposition}

\begin{proof} The first part of this proof uses similar ideas to  \citep[pp.~8-9]{joshaquad}. Since $\deg \varrho = 2$, we have that $ \varrho_*D_i=\varrho_*\varrho^*\Delta_i=2\Delta_i.$ Using this, and the fact that $I^+ \cdot J^+_{ns}(13)(\Q) \subseteq G^+_{13}$, we have \[\varrho_* G = 2 G^+ \supseteq 2 I^+ \cdot J_{ns}^+(13)(\Q).\] We claim that $14I^+ \cdot J_{ns}(13)(\Q) \subseteq G_{13}$. 
Let $D \in J_{ns}(13)(\Q)$. We have \[\varrho_*(2I^+D)=2I^+ \varrho_*D \in 2I^+ \cdot J^+_{ns}(13)(\Q) \subseteq \varrho_*G_{13}.\] So we can choose $D_G \in G_{13}$ satisfying $\varrho_*(2I^+D) = \varrho_* D_G$. So $\varrho^*(2I^+D-D_G) = 0$. Since $J_{ns}^+(13)(\Q)$ and $J^+_{ns}(13)(\Q)$ have equal rank, $\varrho^*$ is injective modulo torsion, so $2I^+D-D_G \in J_{ns}(13)(\Q)_{\mathrm{tors}}$.

We find that \begin{align*} \#J_{ns}(13)(\mathbb{F}_{5}) & = 3 \cdot 7 \cdot 11 \cdot 13^2 \cdot 29, \\ \#J_{ns}(13)(\mathbb{F}_{19}) & = 2^2 \cdot 7 \cdot 83 \cdot 97 \cdot 113 \cdot 883, \end{align*} so by injectivity on torsion, $ J_{ns}(13)(\Q)_{\mathrm{tors}} = 0 $ or $ \Z / 7\Z$. So $14I^+D  = 7D_G \in G$, proving that $14I^+ \cdot J_{ns}(13)(\Q) \subseteq G_{13}$. Moreover, as  $J_{ns}(13)(\Q)_{\mathrm{tors}}= 0$ or $\Z / 7\Z$, it follows that $I \mid 14I^+$.
\end{proof}

In order to test whether or not a prime divides the index $I^+$ we use the following saturation test.  Continuing the notation of Section 5.3, we let $G=\langle D_1, \dots, D_n \rangle$ be a full rank subgroup of the Jacobian $J$ of a curve $X$. Let $l$ be a prime such that either  $l \nmid \#(J(\Q)_{\mathrm{tors}})$, or $J(\Q)_{\mathrm{tors}} = 0$. Let $p$ be a prime of good redution for $X$. We define a  map \begin{align*}
\pi_p : (\Z / l \Z)^n  & \longrightarrow J(\mathbb{F}_p) / lJ(\mathbb{F}_p) \\ (b_1, \dots ,b_n) & \longmapsto  b_1 \widetilde{D_1} + \dots + b_n \widetilde{D_n} ,
\end{align*} where $\widetilde{D_i}$ is the image of $D_i$ in $J(\mathbb{F}_p) / lJ(\mathbb{F}_p)$. Note that if $l \mid \#J(\mathbb{F}_p)$ then $J(\mathbb{F}_p) / lJ(\mathbb{F}_p) \neq 0$.

\begin{proposition} Suppose $ \bigcap_{i=1}^k \ker(\pi_{p_i}) = 0 $ for some choice $p_1, \dots, p_k$ of primes of good reduction for $X$. Then $l \nmid I$.  
\end{proposition}

\begin{proof}
Suppose for a contradiction that $l \mid I$, and choose $Q \in J(\Q) \backslash G$ such that $lQ \in G$. Write $lQ = a_1 D_1 + \dots + a_n D_n$ for some $a_i \in \Z$. For any prime $p$ of good reduction, \[\pi_p(\widetilde{a_1}, \dots,\widetilde{a_n})= \widetilde{a_1}\widetilde{D_1} +\dots + \widetilde{a_n} \widetilde{D_n} = l \widetilde{Q} = 0 \in J(\mathbb{F}_p) / lJ(\mathbb{F}_p).\] So \[(\widetilde{a_1},\dots,\widetilde{a_n}) \in \bigcap_{i=1}^k \ker(\pi_{p_i}) = 0. \] So $l \mid a_i$ for each  $i \in \{1, \dots, n \}$, but then \[l\big(Q-(a_1/l) D_1 + \dots + (a_n/l) D_n\big)=0. \] Since, by assumption, either $l \nmid \#(J(X)(\Q)_{\mathrm{tors}})$, or $J(X)(\Q)_{\mathrm{tors}} = 0$, we must have that 
$Q = (a_1/l) D_1 + \dots + (a_n/l) D_n \in G$, a contradiction. We therefore conclude that $l \nmid I$.
\end{proof}

This gives a way of testing whether $l \nmid I$. We note that even if $l \nmid I$, the method is not guaranteed to show this. To have any chance of success, we must choose primes $p$ so that $l \mid \#J(\mathbb{F}_p)$, as otherwise $\ker(\pi_p)=(\Z / l \Z)^n$. We have implemented this saturation test in \texttt{Magma}. Applying the test to $G_{13}^+ \subseteq J_{ns}^+(13)(\Q)$ we find that \[ 3,5,13,29,41,43,83,97,113,127 \nmid I^+. \] By Proposition \ref{index}, it follows that these primes do not divide $I$ either. In particular, $3,5,13,29 \nmid I$, justifying our choice of the parameter $M$ in the sieving step.

\section{Concluding Remarks}

We start by noting that Theorem \ref{mainthm} concerns only $p=7,11,$ and $13$ (the first primes for which the curve $X_{ns}(p)$ has positive genus). What if $p \geq 17$? We expect that Theorem \ref{mainthm} still holds for $p \geq 17$,  but unfortunately, equations for $X_{ns}(p)$ (if any were known) would be too complicated to carry out explicit computations using our methods. The genus grows quickly, with $X_{ns}(17)$ and $X_{ns}(19)$ having genera $15$ and $20$ respectively. Even equations for $X_{ns}^+(p)$ become difficult to compute and work with for $p \geq 17$. For example, in \citep{modinvariant}, a model for the genus $13$ curve $X_{ns}^+(23)$ is computed, given by $55$ equations. Given these computational issues, it seems as though a more general theoretical argument is needed. We note that it may still be possible to prove that $X_{ns}(p)$ contains no exceptional quadratic points even if $X_{ns}^+(p)(\Q)$ is not fully known. 

It may also be fruitful to consider the analogous question for $X_0(p)$ and $X_0^+(p)$. Exceptional quadratic points do appear for many values of $p$ in this case \citep{hyperquad, joshaquad}, but it may be possible to find a bound for the number of such points.
 
\section*{Appendix}

We display here the $15$ degree $2$ equations for the new model obtained for $X_{ns}(13)$ as a curve in $\mathbb{P}^7$.
\begingroup
\allowdisplaybreaks
\begin{align*} & -2x_1^2 + x_1x_2 - 3x_1x_3 + x_1x_4 - 4x_1x_5 - x_1x_7 + 3x_2^2 + 2x_2x_3 - x_2x_4 
      + 3x_2x_5 + x_2x_7 - x_2x_8  \\ & \qquad-  x_3^2  - 2x_3x_4  + 2x_3x_7 + 5x_3x_8 - x_4x_5 + 
    x_4x_7   - x_4x_8  - x_7^2 - x_7x_8 + 2x_8^2 = 0, \\     
& -3x_1x_2 + 2x_1x_3 - 4x_1x_4 + x_1x_5 + 2x_1x_6 - x_1x_7 + 2x_1x_8 - 3x_2^2 - 
    x_2x_3    + 3x_2x_4 + x_2x_5 + \\ & \qquad 2x_2x_6  + 2x_2x_7 - x_2x_8 + 2x_3^2 + x_3x_5 + 
    2x_3x_6   - 3x_3x_7 + 4x_3x_8 - x_4^2   - 3x_4x_6 + 2x_4x_8 \\ & \qquad  + x_5x_6   - x_5x_7 + 
    3x_5x_8 - x_6x_7  + x_7^2 - x_7x_8 - 2x_8^2 = 0, \\
& -2x_1^2 - 2x_1x_2 - x_1x_3 + 3x_1x_4 + x_1x_5 + 2x_1x_7 - 2x_1x_8 + x_2x_3 - 
    4x_2x_4  + x_2x_5 + 5x_2x_6 \\ & \qquad + 4x_2x_8  + x_3^2  - x_3x_4  + x_3x_5  + x_3x_6 - x_3x_7
    + 7x_3x_8  - x_4^2 - 2x_4x_6 + x_4x_7  + x_4x_8+ x_5x_6  \\ & \qquad  - x_5x_7 + x_5x_8 - 
    2x_7x_8 = 0, \\
& 2x_1^2 - 2x_1x_2 - 5x_1x_3 + x_1x_4 + x_1x_6 + x_1x_7 + x_2x_3 - x_2x_4 + x_2x_6 - 
    x_2x_7  + 2x_2x_8 + 2x_3^2 \\ & \qquad - 2x_3x_4  + x_3x_6 - 2x_3x_7 + 3x_3x_8 - x_4^2 - 
    x_4x_5 + x_4x_7  + x_4x_8 - x_5^2 + x_5x_6  + 2x_5x_8 + 2x_6^2 \\ & \qquad - 2x_6x_7 - x_6x_8 +
    x_7^2 - x_7x_8   - 2x_8^2 =0, \\
& 
-3x_1^2 - 2x_1x_2 + x_2^2 + 4x_2x_3 + 3x_3^2 - x_4^2 + x_4x_5 - x_4x_6 + x_4x_7 + 
    4x_4x_8  + x_5^2  - x_5x_6 + 2x_5x_8 \\ & \qquad + x_6x_7 - 4x_6x_8 - 2x_7x_8 - x_8^2 = 0, \\
& 
-x_1^2 + x_1x_3 - 4x_1x_4 + 2x_1x_5 - x_1x_7 + x_2^2 + 3x_2x_3 + 3x_2x_4 - 3x_2x_5 
    + 2x_2x_7 - 2x_2x_8 + 2x_3^2 \\ & \qquad + x_3x_5 - 3x_3x_7 - 3x_3x_8 - x_4x_6 + x_5^2 - 
    x_5x_6  - x_5x_7 + 2x_5x_8 + x_6x_7  - 2x_6x_8 - x_7x_8  + x_8^2 \\ & \qquad = 0, \\
& -2x_1x_6 - 4x_1x_8 - 2x_2x_4 - 3x_2x_5 + 3x_2x_6 + 3x_2x_7 - 3x_3x_4 + 2x_3x_5 
  - x_3x_6  - 2x_3x_7 + 2x_3x_8  \\ & \qquad  - x_4x_5 - x_4x_6 + 2x_5x_8 - x_6x_7=0, \\
& 
-2x_1^2 - 2x_1x_2 - x_1x_3 - x_1x_4 - x_1x_5 + 2x_1x_6 + 2x_1x_8 + x_2x_3 + 
    2x_2x_4 - x_2x_5 - 3x_2x_6  -  2x_2x_7 \\ & \qquad +  x_3^2 - 3x_3x_4   - x_3x_5 + x_3x_6 - 
    3x_3x_7 - x_3x_8  - x_4^2 - 2x_4x_6 + x_4x_7 + x_4x_8  + x_5x_6 - x_5x_7 \\ & \qquad + x_5x_8 -
    2x_7x_8  = 0, \\ 
& -x_1^2 + x_1x_2 - 4x_1x_3 + x_1x_4 - 2x_1x_6 - 2x_1x_7 + 4x_1x_8 + 2x_2^2 - x_2x_3 
    + 2x_2x_4 - 2x_2x_5  - 2x_2x_6 \\ & \qquad - 3x_2x_7 - 3x_2x_8 - 3x_3^2 - 4x_3x_4 - 
    3x_3x_5  - 2x_3x_6 + 3x_3x_7 - x_4x_5  + x_4x_6 + x_4x_7 - x_4x_8 \\ & \qquad  - x_5^2 + x_5x_6
    + x_5x_7  - 2x_5x_8 - x_6x_7 + 2x_6x_8 - x_7^2 + x_8^2  = 0, \\ 
& -4x_1^2 - 4x_1x_2 - 2x_1x_3 + 2x_2x_3 + 2x_3^2 - 2x_4^2 + 2x_4x_5 - 2x_4x_6 + 
    2x_4x_7  + 2x_4x_8 + 2x_5x_6 \\ & \qquad - 2x_5x_7  - 2x_5x_8 + 2x_6x_7 - 4x_7x_8 = 0, \\ 
& 2x_1x_4 + 2x_1x_6 + 2x_1x_7 - 2x_2x_4 + 2x_2x_6 - 2x_2x_7 + 4x_2x_8 - 4x_3x_4 +
    2x_3x_6   - 4x_3x_7 + 6x_3x_8\\ & \qquad  =0, \\
& 
-x_1^2 + x_1x_2 - 4x_1x_3 + 3x_1x_4 + 2x_1x_5 + 2x_1x_7 + 2x_2^2 - x_2x_3 - 
    2x_2x_4  - 3x_2x_7 + 3x_2x_8  - 3x_3^2 \\ & \qquad + 2x_3x_4 - x_3x_5 + x_3x_7 - 2x_3x_8 - 
    x_4x_5 +  x_4x_6  + x_4x_7 - x_4x_8 - x_5^2 + x_5x_6  + x_5x_7 - 2x_5x_8  \\ & \qquad - x_6x_7 + 
    2x_6x_8  - x_7^2 + x_8^2 = 0, \\
& -2x_1x_4 - x_1x_5 + x_1x_6 - x_1x_7 + 2x_1x_8 + 3x_2x_4 - x_2x_5 - 4x_2x_6 - 
    x_2x_7  - 2x_2x_8 - x_3x_4 - x_3x_5 \\ & \qquad - x_3x_7 - 4x_3x_8 - 2x_4^2 - x_4x_5 + 
    2x_4x_6 + 2x_4x_7  - x_5^2 + x_5x_7 + x_5x_8 - x_6x_8 + 2x_7^2 - 3x_7x_8 \\ & \qquad + 2x_8^2 = 0, \\ 
& x_1^2 - x_1x_2 + 4x_1x_3 + x_1x_5 - 2x_1x_6 + x_1x_7 - 6x_1x_8 + 4x_2^2 + 6x_2x_3 -
    3x_2x_4  - 2x_2x_5 + 3x_2x_6 \\ & \qquad + 4x_2x_7  + x_2x_8 - 3x_3^2 + 2x_3x_4 + 3x_3x_5 -
    x_3x_6 - x_3x_7  - 2x_3x_8 + 2x_4^2  +  x_4x_6 - 2x_4x_7 \\ & \qquad + 3x_5^2   - 2x_5x_6 - 
    x_5x_7 + x_5x_8  - 2x_6^2  - x_6x_7 - 4x_6x_8 - x_7^2 - x_7x_8 + x_8^2 = 0, \\ 
& 
-2x_1x_2 + 10x_1x_3 - 3x_1x_4 + 3x_1x_5 - 2x_1x_6 - x_1x_7 - 2x_1x_8 + 3x_2^2  - 
    16x_2x_3  + x_2x_4 - 2x_2x_5 \\ & \qquad  + 3x_2x_6 + 4x_2x_7 - 2x_2x_8 + 5x_3^2 + 3x_3x_4 
     + 2x_3x_5 - x_3x_6 - 2x_3x_8 + x_4^2 + x_4x_5   - 2x_4x_7 \\ & \qquad + x_4x_8 + x_5x_6 + 
    3x_5x_8  + 2x_6x_7 + 6x_6x_8 + x_7^2 - x_7x_8 + 6x_8^2 = 0. 
\end{align*}
\endgroup

The map $\varrho$ from this new model to the curve $X_{ns}^+(13)$ is given by\begin{align*}
\varrho: X_{ns}(13) & \longrightarrow X_{ns}^+(13) \\ 
(x_1: \dots :x_8) & \longmapsto (-3x_2+x_3 :-2x_1-3x_1+x_3:x_1+x_2+x_3).  \end{align*}

\bibliographystyle{plainnat}

\end{document}